\documentclass{amsart}
\usepackage[unicode=true,pdfusetitle,
bookmarks=true,bookmarksnumbered=false,bookmarksopen=false,
breaklinks=false,pdfborder={0 0 1},backref=false,colorlinks=false]
{hyperref}

\newtheorem{theorem}{Theorem}[section]
\newtheorem{lemma}[theorem]{Lemma}

\theoremstyle{definition}

\theoremstyle{remark}
\newtheorem{remark}[theorem]{Remark}

\numberwithin{equation}{section}

\begin{document}

\title[Uniqueness result for an age-dependent problem]{Uniqueness result for an age-dependent reaction-diffusion
problem}

\author{Vo Anh Khoa}

\address{Institute for Numerical and Applied  Mathematics, University of Goettingen, Germany}

\email{vakhoa.hcmus@gmail.com}

\author{Tran The Hung}

\address{Faculty of Solid State Physics, Gdansk University of Technology, Poland}

\email{thehung.tran@mathmods.eu}

\author{Daniel Lesnic}

\address{Department of Applied Mathematics, University of Leeds, United Kingdom}

\email{amt5ld@maths.leeds.ac.uk}


\begin{abstract}
This paper is concerned with an age-structured model in population
dynamics. We investigate the  uniqueness of solution for this type of nonlinear
reaction-diffusion problem when the source term depends on the density, indicating
the presence of, for example, mortality and reaction processes.
Our result shows that in a spatial environment, if two population
densities obey the same evolution equation and possess the same terminal data of time and age,
then their distributions must coincide therein.
\end{abstract}

\maketitle

\section{Introduction}

In population dynamics, there are several factors interestingly contributing
to the complicated and nontrivial spatio-temporal spread patterns of
diseases. Especially demographic-disease ages and spatial movement
naturally interweave in many ways. Their correlation and interaction
are expected to participate in mathematical modeling and analysis,
and can lead to models with distinct levels of complexity.

As far as we know, disease management measures are often age-dependent
and their effectiveness may also be dependent on the mobility status
of the involved species, such as larvicides and insecticide sprays
for mosquito-borne diseases (e.g. the invasion and spread of West
Nile virus in North America in \cite{MR08}) used to control the mosquito
population at different levels of their maturity. We recall the structured
population model in \cite{SZ01} for the total number of matured individuals
in a single species population at the demographic age, which reads as
\begin{align}
u_{t}+u_{a}=D\Delta u-d_0 u.\label{eq:single}
\end{align}
Here, $t$ denotes the time, $a$ denotes age, $u$ represents the population density, $D$ and $d_0$ are, respectively,
the diffusion and death rates of the immature individual, under the assumption that the maturation rate is regulated by the birth process and the dynamics of the individual during the maturation phase.

A generalized approach for this mosquito-borne disease model is given by the coupled McKendrick von-Foerster equations (see also \cite{Mc26,Foer59,Wit81})
where the reservoir as the host and the age-structured host population
are observed. Following \cite{GLW07}, the
spatial spread of vector-borne diseases model involving age-structure
is
\begin{equation}
\begin{cases}
u_{t}+u_{a} =D_{1}\left(a\right)\Delta u-u\left(d_{1}\left(a\right)+\chi\left(a\right)\text{m}\left(t,x\right)\right), &\\
v_{t}+v_{a} =D_{2}\left(a\right)\Delta v-vd_{2}\left(a\right)+u\chi\left(a\right)\text{m}\left(t,x\right).&
\end{cases}
\end{equation}
Here, $u$ and $v$ play
roles as susceptible and infected hosts, respectively, while $D_{i}\left(a\right)$
are the age-dependent diffusivities by natural means, $\chi (a)$ is the age-dependent transmission coefficient, $d_{i}\left(a\right)$
are the age-dependent death rates of such hosts for $i=1,2$
and $\text{m}$ can be regarded as the number of infected adult mosquitoes.


Let $\Omega\subset\mathbb{R}^{n}$ for $n\ge 2$ be a bounded, open and connected
domain with sufficiently smooth boundary and let $\mathcal{A}$ be
the linear second-order differential operator
\begin{equation}
\mathcal{A}u\left(t,a,x\right)=\sum_{i,j=1}^{n}\frac{\partial}{\partial x_{i}}\left(d_{ij}\left(t,a,x\right)\frac{\partial u\left(t,a,x\right)}{\partial x_{j}}\right),
\end{equation}
accounting e.g. for the anisotropic diffusion and possibly taxis processes.
 
Setting $Q_{T,a_{\dagger}}^{\Omega}:=\left(0,T\right)\times\left(0,a_{\dagger}\right)\times\Omega$
for $T,a_{\dagger}>0$, we consider the problem:
\begin{equation}
\begin{cases}
u_{t}+u_{a}-\mathcal{A}u=F\left(t,a,x;u\right) & \text{in }Q_{T,a_{\dagger}}^{\Omega},\\
u\left(T,a,x\right)=u_{0}\left(a,x\right) & \text{in }\left(0,a_{\dagger}\right)\times\Omega,\\
u\left(t,a_{\dagger},x\right)=u_{1}\left(t,x\right) & \text{in }\left(0,T\right)\times\Omega,
\end{cases}\label{eq:P}
\end{equation}
where $F$ is a given source term, endowed either with the homogeneous Dirichlet boundary condition
\begin{align}
u\left(t,a,x\right)=0\quad  \text{ on }\left(0,T\right)\times\left(0,a_{\dagger}\right)\times\partial\Omega,
\label{Dirichlet}
\end{align}
or with the Robin-type boundary condition
\begin{align}
- \text{d} (t,a,x)\nabla u\left(t,a,x\right)\cdot\text{n}=S(u)\quad  \text{ on }\left(0,T\right)\times\left(0,a_{\dagger}\right)\times\partial\Omega.
\label{Robin}
\end{align}

Here, $\text{n}$ denotes the outward unit normal vector to the boundary $\partial\Omega$, $\text{d} =\left(d_{ij}\right)_{i,j=\overline{1,n}}$ is the diffusion tensor and $S$ is some given function.

In this paper, we prove that under certain assumptions on the source
term $F$ and surface reaction $S$, the solution to the system \eqref{eq:P} and \eqref{Dirichlet} or \eqref{Robin} is unique, if it exists. The present source term $F$ can model several types of polynomial reactions (e.g. logistic and von
Bertalanffy) and even exponential growth (Arrhenius laws). 
It can also be modelled as 
\[ F\left(t,a,x;u\right)=-\tilde{\mu}\left(t;u\right)u\left(t,a,x\right),
\]
where $\tilde{\mu}>0$ is called the time- and density-dependent mortality
modulus. This mortality-related functional usually arises in the Lotka-von
Foerster model, where the simple modulus is
\[
\tilde{\mu}\left(t;u\right)=\int_{0}^{a_{\dagger}}\int_{\Omega}u\left(t,a,x\right)dxda,
\]
provided the mortality process is also controlled by the total population
at time $t$ during the whole age and environment.

It is worth noting that when age can be viewed as a temporal time,
the first equation in \eqref{eq:P} is usually referred to as an ultra-parabolic
equation. In literature, the backward problem \eqref{eq:P} and \eqref{Dirichlet} has
been explored in \cite{ZR14} for an abstract linear class
of ultra-parabolic problems. Such problems are basically ill-posed
in the sense that the solution does not depend continuously on the
data no matter how smooth it is. As a matter of fact, a regularization
has to be designed. The starting point of an evolution equation involving multi-time variables is from \cite{Lorenzi1,Lorenzi2}. It is shown in \cite{Lorenzi1} that
the forward problem for \eqref{eq:P} and \eqref{Dirichlet} is well-posed in the space of H\"older continuous functions. As a byproduct, the
result therein provides the representation of solution, based
on the semi-group theory along the upper and lower triangles dividing the rectangle of times. The evolution of the system with two different
times were discussed by an argument where some diffusion processes
with memory take place. In the framework of stochastic optimal control,
the reader can be referred to e.g. \cite{Marco03,Mar09}, for related contributions to this problem.

The rest of this paper is organized as follows: Section \ref{sec:Preliminaries} is devoted to introduce notation and conventions throughout this paper. In addition, we provide technical assumptions on parameters and coefficients involved in \eqref{eq:P}. In Section \ref{sec:main}, we deliberately present two subsections where the Dirichlet and Robin-type boundary conditions are considered, respectively. Our main results are reported in Theorem \ref{thm:main1} and Theorem \ref{thm:main2}, whilst their cores of proof are based on Lemma \ref{lem:3.1-1} and Lemma \ref{lem:3.2-1}, respectively.   Essentially, the main technique here is inspired from \cite[Chapter 6]{Av83}. This approach was used to treat the backward parabolic operator within the study of the large-time behavior of solutions to a linear class of initial-boundary value parabolic equations. It is also helpful in the analysis of regularization for inverse and ill-posed problems. In this work, we extend the uniqueness result particularly to a class of semi-linear age-dependent reaction-diffusion equations.

\section{Preliminaries\label{sec:Preliminaries}}

In parallel with using the notation $Q_{T,a_{\dagger}}^{\Omega}$,
the same meaning is also given to the notation $Q_{T,a_{\dagger}}=\left(0,T\right)\times\left(0,a_{\dagger}\right)$
and $Q_{\left(t_{1},t_{2}\right)\times\left(a_{1},a_{2}\right)}^{\Omega}=\left(t_{1},t_{2}\right)\times\left(a_{1},a_{2}\right)\times\Omega$,
etc. throughout this paper. Moreover, depending on the context, by $\left|\cdot\right|$ we denote either the volume measure of a domain or the absolute value of a function.

To this end, for any domain  $D$ we also set $\left\Vert \cdot\right\Vert _{D}$ to be the norm in $L^{2}\left(D\right)$ and the same indication goes to the inner product $\left\langle \cdot,\cdot\right\rangle _{D}$.

For the sake of our analysis in this work, we need the following
set of assumptions:

$\left(\text{A}_{1}\right)$ The diffusion tensor $\text{d}=\left(d_{ij}\right)_{1\le i,j\le n}\in \left[C^{1}\left(\overline{Q_{T,a_{\dagger}}^{\Omega}}\right)\right]^{n\times n}$
is symmetric and there exist positive constants $\underline{c}$ and $\overline{c}$
such that
\[
\underline{c}\left|\xi\right|^{2}\le\sum_{i,j=1}^{n}d_{ij}\left(t,a,x\right)\xi_{i}\xi_{j}\le\overline{c}\left|\xi\right|^{2}\quad\text{for any }\xi\in\mathbb{R}^{n}.
\]

$\left(\text{A}_{2}\right)$ There exists a positive constant $\overline{M}>0$
such that
\[
\left|\partial_{t}d_{ij}\left(t,a,x\right)\right| + \left|\partial_{a}d_{ij}\left(t,a,x\right)\right|\le\overline{M}\quad\text{for all }\left(t,a,x\right)\in\overline{Q_{T,a_{\dagger}}^{\Omega}}.
\]

$\left(\text{A}_{3}\right)$ The source term $F:\overline{Q_{T,a_{\dagger}}^{\Omega}}\times L^2\left(Q^{\Omega}_{T,a_{\dagger}}\right)\to\mathbb{R}$
is $\left(0,1\right] \ni\alpha$-H\"older continuous with respect to the density argument,
i.e. there exists $L_{F}>0$ such that for all $\left(t,a,x\right)\in\overline{Q_{T,a_{\dagger}}^{\Omega}}$,
\[
\left|F\left(t,a,x;u_1\right)-F\left(t,a,x;u_2\right)\right|\le L_{F}\left| u_1-u_2\right|^{\alpha}\text{ for all } u_1,u_2\in L^2\left(Q^{\Omega}_{T,a_{\dagger}}\right).
\]

$\left(\text{A}_{4}\right)$ For the surface reaction term $S:L^2 \left(\partial \Omega\right)\to\mathbb{R}$ there exists $\overline{m}>0$ such that
\[
\left\langle \left(\partial_{t}+\partial_{a}\right)\left(S\left(u_1\right)-S\left(u_2\right)\right),u_1-u_2\right\rangle _{\partial\Omega}\le\overline{m}\left\Vert u_1-u_2\right\Vert _{\partial\Omega}^{2}\; \text{ for all } u_1,u_2\in L^2 \left(\partial \Omega\right).
\]

$\left(\text{A}_{5}\right)$ The surface reaction term $S:L^2 \left(\partial \Omega\right)\to\mathbb{R}$
is monotone, i.e. 
\[
\left\langle S\left(u_1\right)-S\left(u_2\right),u_1-u_2\right\rangle _{\partial\Omega}\ge0 \quad \text{for all } u_1,u_2\in L^2 \left(\partial \Omega\right).
\]

$\left(\text{A}_{6}\right)$ The surface reaction term $S:L^2 \left(\partial \Omega\right)\to\mathbb{R}$ is
 $\left(0,1\right]\ni\beta$-H\"older continuous, i.e. there exists $L_{S}>0$ such that
\[
\left|S\left(u_1\right)-S\left(u_2\right)\right|\le L_{S}\left| u_1-u_2\right|^{\beta} \quad\text{for all } u_1,u_2\in L^2 \left(\partial \Omega\right).
\]

\begin{lemma}
(cf. \cite{JM07}) For any $\gamma>0$ and $\alpha_0 \in\left(0,1\right]$,
the following inequality holds
\begin{equation}
X^{\alpha_0}\le\alpha_0 \gamma^{\alpha_0 -1}X+\left(1-\alpha_0 \right)\gamma^{\alpha_0}\quad\text{for all }X\ge0.\label{eq:bdtphu}
\end{equation}
\end{lemma}

\begin{lemma}\label{lem:trace}
(cf. \cite{Hor79}) There exists a positive constant $C_0$ such that
\begin{equation}
\left\Vert u\right\Vert _{\partial\Omega}^2 \leq C_0 \left(\left\Vert u\right\Vert^2 _{\Omega}+\left\Vert \nabla u\right\Vert^2 _{\Omega}\right)\quad\text{for any }u\in H^1(\Omega).
\label{eq:bdt_trace}
\end{equation}
\end{lemma}

\begin{remark}
	The constant $C_0$ in \eqref{eq:bdt_trace} can be identified from the trace inequality
	\begin{align}
	\left\Vert u\right\Vert _{\partial\Omega}^{2}\le c\left(\varepsilon,\Omega \right)\left\Vert u\right\Vert _{\Omega}^{2}+\varepsilon\left\Vert \nabla u\right\Vert _{\Omega}^{2}\quad\text{for any }u\in H^{1}\left(\Omega\right)\;\text{and }\varepsilon>0.
	\label{eq:bdt_tracephu}
	\end{align}
This inequality is roughly dependent on the geometry of $\Omega$. As a typical example from \cite[Theorem 2]{Hor79}, if $\Omega$ is star-shaped, $c(\varepsilon,\Omega)$ is of the order of $\mathcal{O}(1+\varepsilon^{-1})$. When $u$ has zero mean on $\partial \Omega$, i.e. $\int_{\partial \Omega} ud\sigma_{x} = 0$, then we have the stronger Poincar\'{e} inequality
$\left\Vert u\right\Vert _{\partial\Omega}^2 \leq \tilde{C}_0 \left\Vert \nabla u\right\Vert^2 _{\Omega}$, for some constant $\tilde{C}_0>0$.
\end{remark}

\section{Main results}\label{sec:main}

\subsection{Dirichlet boundary condition \eqref{Dirichlet} }\label{subsec:31}

Let us set the function space 
\begin{align}
W_{T,a_{\dagger}}^{\Omega}:=C\left(\overline{Q_{T,a_{\dagger}}};H_{0}^{1}\left(\Omega\right)\right)\cap L^{2}\left(Q_{T,a_{\dagger}};H^{2}\left(\Omega\right)\right)\cap C^{1}\left(Q_{T,a_{\dagger}};H_{0}^{1}\left(\Omega\right)\right).
\label{W1}
\end{align}

Assume that there exist two solutions $u_1$ and $u_2$ of \eqref{eq:P} and \eqref{Dirichlet} that belong to $W_{T,a_{\dagger}}^{\Omega}$. Then
their difference $w=u_1-u_2\in W_{T,a_{\dagger}}^{\Omega}$ satisfies
\begin{equation}
\begin{cases}
w_{t}+w_{a}-\mathcal{A}w=F\left(t,a,x;u_1\right)-F\left(t,a,x;u_2\right) & \text{in }Q_{T,a_{\dagger}}^{\Omega},\\
w\left(t,a,x\right)=0 & \text{on }Q_{T,a_{\dagger}}^{\partial\Omega},\\
w\left(T,a,x\right)=0 & \text{in }\left(0,a_{\dagger}\right)\times\Omega,\\
w\left(t,a_{\dagger},x\right)=0 & \text{in }\left(0,T\right)\times\Omega.
\end{cases}\label{eq:Pnew}
\end{equation}

Let us denote
\[
P_{T,a_{\dagger}}^{\Omega}:=\left\{ u\in W_{T,a_{\dagger}}^{\Omega}:\left.u\right|_{\partial\Omega}=0,\left.u\right|_{t\in\left\{ 0,T\right\} }=0,\left.u\right|_{a\in\left\{ 0,a_{\dagger}\right\} }=0\right\}
\]
and $\lambda\left(t,a\right):=t-T+a-a_{\dagger}-\eta<0$
for some $\eta>0$.
\begin{lemma}
\label{lem:3.1-1}Assume $\left(\text{A}_{1}\right)$ and $\left(\text{A}_{2}\right)$ hold. Then for any $v\in P_{T,a_{\dagger}}^{\Omega}$, $m\in \mathbb{N}^{*}$ and $  k\in \mathbb{R}^{*}_{+}$, we have
\begin{align}
\left\Vert \lambda^{-\frac{m}{k}}\left(\mathcal{A}v-v_{t}-v_{a}\right)\right\Vert _{Q_{T,a_{\dagger}}^{\Omega}}^{2} \geq\dfrac{4m}{k}\left\Vert \lambda^{-\frac{m}{k}-1}v\right\Vert _{Q_{T,a_{\dagger}}^{\Omega}}^{2}
 -\overline{M}\left\Vert \lambda^{-\frac{m}{k}}\nabla v\right\Vert _{Q_{T,a_{\dagger}}^{\Omega}}^{2}.\label{eq:3.1-1}
\end{align}
Moreover, if $0<T+a_{\dagger}\le\mu$ for a sufficiently small constant $\mu>0$, there exists a positive $K$ such that
\begin{equation}
K\left\Vert \lambda^{-\frac{m}{k}}\left(\mathcal{A}v-v_{t}-v_{a}\right)\right\Vert _{Q_{T,a_{\dagger}}^{\Omega}}^{2}\ge\left\Vert \lambda^{-\frac{m}{k}-1}v\right\Vert _{Q_{T,a_{\dagger}}^{\Omega}}^{2}+\frac{1}{2}\left\Vert \lambda^{-\frac{m}{k}}\nabla v\right\Vert _{Q_{T,a_{\dagger}}^{\Omega}}^{2},\label{eq:3.3-1}
\end{equation}
for sufficiently large $m$. 
\end{lemma}

\begin{proof}
Let $z=\lambda^{-\frac{m}{k}}v\in P_{T,a_{\dagger}}^{\Omega}$. Then
\begin{equation}
\lambda^{-\frac{m}{k}}\left(\mathcal{A}v-v_{t}-v_{a}\right)=\mathcal{A}z-\left(z_{t}+z_{a}\right)-\dfrac{2m}{k}\lambda^{-1}z.\label{eq:3.3}
\end{equation}

By the definition of inner product in $L^{2}\left(Q_{T,a_{\dagger}}^{\Omega}\right)$,
we obtain
\begin{align}
\left\Vert \lambda^{-\frac{m}{k}}\left(\mathcal{A}v-v_{t}-v_{a}\right)\right\Vert _{Q_{T,a_{\dagger}}^{\Omega}}^{2} & =\left\Vert z_{t}+z_{a}\right\Vert _{Q_{T,a_{\dagger}}^{\Omega}}^{2}-2\left\langle z_{t}+z_{a},\mathcal{A}z\right\rangle _{Q_{T,a_{\dagger}}^{\Omega}}\nonumber \\
& +\dfrac{4m}{k}\left\langle \lambda^{-1}z,z_{t}+z_{a}\right\rangle _{Q_{T,a_{\dagger}}^{\Omega}} +\left\Vert \mathcal{A}z-\dfrac{2m}{k}\lambda^{-1}z\right\Vert _{Q_{T,a_{\dagger}}^{\Omega}}^{2}.\label{eq:3.3.1}
\end{align}
\noindent
For the second term in right-hand side (RHS) of \eqref{eq:3.3.1}, using  integration by parts
\begin{align*}\label{eq:3.3.21}
-2\left\langle z_{t}+z_{a},\mathcal{A}z\right\rangle _{Q_{T,a_{\dagger}}^{\Omega}} & =2\int_{0}^{T}\int_{0}^{a_{\dagger}}\int_{\Omega}\sum_{1\leq i,j\leq n}\left(\partial_{t}\partial_{x_{i}}z+\partial_{a}\partial_{x_{i}}z\right)d_{ij}\partial_{x_{j}}zdxdadt\\
& = 2\int_{0}^{T}\int_{0}^{a_\dagger}\int_{\Omega}(\partial_{t} + \partial_{a})\left(\sum_{1\leq i,j\leq n}\partial_{x_i}z d_{ij} \partial_{x_j} z \right)dxdadt\\ &
- 2 \int_{0}^{T}\int_{0}^{a_\dagger}\int_{\Omega}\sum_{1\leq i,j\leq n}\partial_{x_i}z (\partial_{t} + \partial_{a})\left( d_{ij} \partial_{x_j} z\right)dxdadt \\ &
= -2 \int_{0}^{T}\int_{0}^{a_\dagger}\int_{\Omega} \sum_{1\leq i,j\leq n} \partial_{x_i} z (\partial_{t} + \partial_{a}) \left(d_{ij}\right) \partial_{x_j} z dxdadt \\ &
-2 \int_{0}^{T}\int_{0}^{a_\dagger}\int_{\Omega} \sum_{1\leq i,j\leq n} \partial_{x_{i}} z d_{ij} (\partial_{t} + \partial_{a}) \partial_{x_{j}} z dxdadt,
\end{align*}
where we have used that $z\in P_{T,a_\dagger}^{\Omega}$. Since $d_{ij} = d_{ji}$ (cf. $\left(\text{A}_1 \right)$), from the first and fifth rows we obtain 
\begin{align*}
 & 2 \int_{0}^{T}\int_{0}^{a_{\dagger}}\int_{\Omega}\sum_{1\leq i,j\leq n} \left(\partial_{t}+\partial_{a}\right)\left(\partial_{x_{i}}z \right) d_{ij} \partial_{x_{j}}z dxdadt\nonumber \\
& = - \int_{0}^{T}\int_{0}^{a_{\dagger}}\int_{\Omega}\sum_{1\leq i,j\leq n}\partial_{x_{i}}z \left(\partial_{t} + \partial_{a}\right)\left(d_{ij}\right)\partial_{x_{j}}z dxdadt .
\end{align*}

Thus, we obtain that
\begin{equation}
-2\left\langle z_{t}+z_{a},\mathcal{A}z\right\rangle _{Q_{T,a_{\dagger}}^{\Omega}} = -\int_{0}^{T}\int_{0}^{a_{\dagger}}\int_{\Omega}\sum_{1\leq i,j\leq n}\partial_{x_{i}}z\left(\partial_{t} + \partial_{a}\right) \left(d_{ij}\right)\partial_{x_{j}}z dxdadt.
\label{eq:3.3.2}
\end{equation}

In the same manner, the third term in the RHS of \eqref{eq:3.3.1} is
\begin{align}
\dfrac{4m}{k}\left\langle \lambda^{-1}z,z_{t}+z_{a}\right\rangle _{Q_{T,a_{\dagger}}^{\Omega}} &  =\frac{2m}{k}\int_{0}^{T}\int_{0}^{a_{\dagger}}\int_{\Omega}\left(\partial_{t} + \partial_{a}\right)\left(\lambda^{-1}z^{2}\right)dxdadt \nonumber\\
& +\frac{4m}{k}\int_{0}^{T}\int_{0}^{a_{\dagger}}\int_{\Omega}\lambda^{-2}z^{2}dxdadt  =\dfrac{4m}{k}\left\Vert \lambda^{-1}z\right\Vert _{Q_{T,a_{\dagger}}^{\Omega}}^{2}. 
\label{eq:3.3.3}
\end{align}

Using $\left(\text{A}_2\right)$, from \eqref{eq:3.3.2} and \eqref{eq:3.3.3}, we obtain
\begin{eqnarray*}
\left\Vert \lambda^{-\frac{m}{k}}\left(\mathcal{A}v-v_{t}-v_{a}\right)\right\Vert _{Q_{T,a_{\dagger}}^{\Omega}}^{2} \geq\dfrac{4m}{k}\left\Vert \lambda^{-1}z\right\Vert _{Q_{T,a_{\dagger}}^{\Omega}}^{2}\\
-\overline{M}\int_{0}^{T}\int_{0}^{a_{\dagger}}\int_{\Omega}\sum_{1\leq i,j\leq n} \left|\partial_{x_{i}}z\partial_{x_{j}}z \right| dxdadt
\geq\dfrac{4m}{k}\left\Vert \lambda^{-1}z\right\Vert _{Q_{T,a_{\dagger}}^{\Omega}}^{2}-\overline{M}\left\Vert \nabla z\right\Vert _{Q_{T,a_{\dagger}}^{\Omega}}^{2}.
\end{eqnarray*}
Substituting $z=\lambda^{-\frac{m}{k}}v$, we end the proof of \eqref{eq:3.1-1}.

Using ($\text{A}_1$), integration by parts and $\left|\lambda\left(t,a\right)\right|\le T+a_{\dagger}+\eta$ we have 
\begin{align}
-\left\langle \lambda^{-\frac{2m}{k}}v,\mathcal{A}v-v_{t}-v_{a}\right\rangle _{Q_{T,a_{\dagger}}^{\Omega}} &   \ge -\frac{2m}{k}\left(T+a_{\dagger}+\eta\right)\left\Vert \lambda^{-\frac{m}{k}-1}v\right\Vert _{Q_{T,a_{\dagger}}^{\Omega}}^{2} \nonumber\\
& +\underline{c}\left\Vert \lambda^{-\frac{m}{k}}\nabla v\right\Vert _{Q_{T,a_{\dagger}}^{\Omega}}^{2}
\label{eq:3.7}
\end{align}
It then follows from \eqref{eq:3.1-1} that \eqref{eq:3.7} can be
estimated by
\begin{align}-\left\langle \lambda^{-\frac{2m}{k}}v,\mathcal{A}v-v_{t}-v_{a}\right\rangle _{Q_{T,a_{\dagger}}^\Omega} & \ge -\frac{\left(T+a_{\dagger}+\eta\right)}{2}\left(\left\Vert \lambda^{-\frac{m}{k}}\left(\mathcal{A}v-v_{t}-v_{a}\right)\right\Vert _{Q_{T,a_{\dagger}}^{\Omega}}^{2}\right. \nonumber\\
& \left.+\overline{M}\left\Vert \lambda^{-\frac{m}{k}}\nabla v\right\Vert _{Q_{T,a_{\dagger}}^{\Omega}}^{2}\right) + \underline{c}\left\Vert \lambda^{-\frac{m}{k}}\nabla v\right\Vert _{Q_{T,a_{\dagger}}^{\Omega}}^{2}.
\label{eq:3.8}
\end{align}

Using \eqref{eq:3.1-1}, we have
\begin{eqnarray} 
-\left\langle \lambda^{-\frac{2m}{k}}v,\mathcal{A}v-v_{t}-v_{a}\right\rangle _{Q_{T,a_{\dagger}}^{\Omega}}
\le\frac{1}{2}\left\Vert \lambda^{-\frac{m}{k}-1}v\right\Vert _{Q_{T,a_{\dagger}}^{\Omega}}^{2}+\frac{1}{2}\left\Vert \lambda^{-\frac{m}{k}+1}\left(\mathcal{A}v-v_{t}-v_{a}\right)\right\Vert _{Q_{T,a_{\dagger}}^{\Omega}}^{2}\nonumber\\
\le\frac{k}{8m}\left(\left\Vert \lambda^{-\frac{m}{k}}\left(\mathcal{A}v-v_{t}-v_{a}\right)\right\Vert _{Q_{T,a_{\dagger}}^{\Omega}}^{2}+\overline{M}\left\Vert \lambda^{-\frac{m}{k}}\nabla v\right\Vert _{Q_{T,a_{\dagger}}^{\Omega}}^{2}\right)\nonumber\\
+\frac{\left(T+a_{\dagger}+\eta\right)^{2}}{2}\left\Vert \lambda^{-\frac{m}{k}}\left(\mathcal{A}v-v_{t}-v_{a}\right)\right\Vert _{Q_{T,a_{\dagger}}^{\Omega}}^{2}.\nonumber
\end{eqnarray}
Combining this with \eqref{eq:3.8} then reads as
\begin{align*}
& \underline{c}\left\Vert \lambda^{-\frac{m}{k}}\nabla v\right\Vert _{Q_{T,a_{\dagger}}^{\Omega}}^{2} \\
& \le\left(\frac{k}{8m}+\frac{\left(T+a_{\dagger}+\eta\right)}{2}+\frac{\left(T+a_{\dagger}+\eta\right)^{2}}{2}\right)\left\Vert \lambda^{-\frac{m}{k}}\left(\mathcal{A}v-v_{t}-v_{a}\right)\right\Vert _{Q_{T,a_{\dagger}}^{\Omega}}^{2}\\
 & +\left(\frac{\left(T+a_{\dagger}+\eta\right)\overline{M}}{2}+\frac{k\overline{M}}{8m}\right)\left\Vert \lambda^{-\frac{m}{k}}\nabla v\right\Vert _{Q_{T,a_{\dagger}}^{\Omega}}^{2}.
\end{align*}

Accordingly, if $\mu_{0}\ge\mu\ge T+a_{\dagger}>0$, $\eta_{0}\ge\eta>0$
and $m_{0}$ are such that
\begin{align}
2\left(\frac{\left(\mu_{0}+\eta_{0}\right)\overline{M}}{2}+\frac{k\overline{M}}{8m_{0}}\right)\leq\underline{c},
\label{eq:3.9.1}
\end{align}
then for any $m\ge m_{0}$, we obtain
\begin{align}
\left\Vert \lambda^{-\frac{m}{k}}\nabla v\right\Vert _{Q_{T,a_{\dagger}}^{\Omega}}^{2} & \le C_1\left\Vert \lambda^{-\frac{m}{k}}\left(\mathcal{A}v-v_{t}-v_{a}\right)\right\Vert _{Q_{T,a_{\dagger}}^{\Omega}}^{2},\label{eq:3.10}
\end{align}
where $C_1 := \frac{2}{\underline{c}}\left(\frac{k}{8m}+\frac{\left(\mu+\eta\right)}{2}+\frac{\left(\mu+\eta\right)^{2}}{2}\right)$. Combining \eqref{eq:3.10} and \eqref{eq:3.1-1}, we conclude that
\begin{align*}
&  \left\Vert \lambda^{-\frac{m}{k}-1}v\right\Vert _{Q_{T,a_{\dagger}}^{\Omega}}^{2} +\frac{1}{2}\left\Vert \lambda^{-\frac{m}{k}}\nabla v\right\Vert _{Q_{T,a_{\dagger}}^{\Omega}}^{2}\\
& \leq\left[\frac{k}{4m} + C_1\left(\frac{1}{2}+ \frac{k\overline{M}}{4m}\right)\right]\left\Vert \lambda^{-\frac{m}{k}}\left(\mathcal{A}v-v_{t}-v_{a}\right)\right\Vert _{Q_{T,a_{\dagger}}^{\Omega}}^{2}.
\end{align*}
Then denoting  $C_2 = \frac{2}{\underline{c}}\left(\frac{k}{8m_0}+\frac{\left(\mu_0+\eta_0\right)}{2}+\frac{\left(\mu_0+\eta_0\right)^{2}}{2}\right)$ and choosing
\begin{equation}
K := \frac{k}{4m_0} + C_2 \left(\frac{1}{2} + \frac{k\overline{M}}{4m_0}\right),
\label{eq:3.14}
\end{equation}
we conclude that  \eqref{eq:3.3-1} holds for $m\ge m_0$.
\end{proof}
Let us now choose $m_0$ sufficiently large, and $\mu_0$ and $\eta_0$ sufficiently small such that $K$ given by \eqref{eq:3.14} satisfies
\begin{equation}
0<K\leq \frac{1}{4\alpha L_F^2}.
\label{eq:3.15}
\end{equation}
Let also $0<\eta_1 \leq \min\{1,\eta_0\}$ and choose 
\begin{equation}
\eta=\eta_{1},\;\mu_{0}'=\eta^{\frac{m\left(\frac{1}{\alpha}-1\right)}{m\left(\frac{1}{\alpha}-1\right)+k}}2^{\frac{k}{2m\left(\frac{1}{\alpha}-1\right)+2k}}-\eta>0.
\label{eq:choice}
\end{equation}

We consider here $T+a_{\dagger}\le\mu_{0}'$ since the uniqueness
result for the latter case $T+a_{\dagger}>\mu_{0}'$ is a direct consequence
from the former case. In that case, the time and age intervals can be divided into many countable subsets whose lengths are not larger than $\mu_{0}'$, which brings us back to the former case.

Let $0<t_{1}<t_{2}<T$ and $0<a_{1}<a_{2}<a_{\dagger}$
and take $\kappa:\overline{Q_{T,a_{\dagger}}}\to\mathbb{R}$ such
that it is twice continuously differentiable in $\left(t_1, t_2\right) \times \left(a_1, a_2\right)$ and
\begin{equation}
\kappa\left(t,a\right)=\begin{cases}
0 & \text{if }\left(t,a\right)\in Q_{T,a_{\dagger}}\backslash\left( \left(t_{1},t_{2}\right)\times\left(a_{1},a_{2}\right) \cup\left[t_{2},T\right]\times\left[a_{2},a_{\dagger}\right]\right),\\
1 & \text{if }\left(t,a\right)\in\left[t_{2},T\right]\times\left[a_{2},a_{\dagger}\right].
\end{cases}\label{eq:kappa}
\end{equation}

Let $v=\kappa w$, where $w$ is the solution of \eqref{eq:Pnew}, then notice that $v\in P_{T,a_{\dagger}}^{\Omega}$. Starting from the estimate \eqref{eq:3.3-1}, we have that for sufficiently large $m\geq m_0$
\begin{align}
K\left\Vert \lambda^{-\frac{m}{k}}\left(\mathcal{A}v\right.\right. & \left. \left. -\; v_{t}-v_{a}\right)\right\Vert _{Q_{\left(t_{1},t_{2}\right)\times\left(a_{1},a_{2}\right)}^{\Omega}}^{2}   +K\left\Vert \lambda^{-\frac{m}{k}}\left(\mathcal{A}w-w_{t}-w_{a}\right)\right\Vert _{Q_{\left(t_{2},T\right)\times\left(a_{2},a_{\dagger}\right)}^{\Omega}}^{2} \nonumber \\
& \ge\left\Vert \lambda^{-\frac{m}{k}-1}w\right\Vert _{Q_{\left(t_{2},T\right)\times\left(a_{2},a_{\dagger}\right)}^{\Omega}}^{2}
+\frac{1}{2}\left\Vert \lambda^{-\frac{m}{k}}\nabla v\right\Vert _{Q_{\left(t_{2},T\right)\times\left(a_{2},a_{\dagger}\right)}^{\Omega}}^{2}.
\label{eq:3.13-1}
\end{align}

From \eqref{eq:Pnew} and ($\text{A}_3$), we get
$\lambda^{-\frac{2m}{k}} \left(\mathcal{A}w - w_t - w_a\right)^2 \leq \lambda^{-\frac{2m}{k}} L_F^2 \left|w\right|^{2\alpha}$, 
then applying Lemma \ref{eq:bdtphu} with $X=\left|w\right|^2 >0 $, $\alpha_0 = \alpha$ and
\begin{equation}
\gamma = \left[(t_2 - t_1)(a_2 - a_1)\right]^{-\frac{1}{\alpha}} \lambda^{\frac{2m}{k\alpha}}
\lVert \lambda^{-\frac{m}{k}} \rVert_{(t_1,t_2)\times (a_1,a_2)}^{\frac{2}{\alpha}} > 0, \label{eq:3.19.1}
\end{equation}
we have
\begin{align*}
\lambda^{-\frac{2m}{k}} \left(\mathcal{A}w - w_t - w_a\right)^2 & \leq \alpha L_F^2 \frac{\lambda^{-\frac{2m}{k\alpha}}\lVert \lambda^{-\frac{m}{k}} \rVert^{\frac{2}{\alpha}(\alpha - 1)}_{(t_1,t_2)\times (a_1,a_2)}}{\left[(t_2 - t_1)(a_2 - a_1)\right]^{1-\frac{1}{\alpha}}} w^2 \\
& + (1-\alpha) L_F^2 \frac{\lVert \lambda^{-\frac{m}{k}} \rVert^2_{(t_1,t_2)\times (a_1,a_2)}}{(t_2 - t_1)(a_2 - a_1)}.
\end{align*}
Integrating both sides over $\Omega$, we have 
\begin{align}
\left\Vert \lambda^{-\frac{m}{k}}\left(\mathcal{A}w-w_{t}-w_{a}\right)\right\Vert _{\Omega}^{2} & \leq 
\alpha L_F^2 \dfrac{\lambda^{-\frac{2m}{k\alpha}+\frac{2m}{k}}\lVert \lambda^{-\frac{m}{k}} \rVert^{\frac{2}{\alpha}(\alpha - 1)}_{(t_1,t_2)\times (a_1,a_2)}}{\left[(t_2 - t_1)(a_2 - a_1)\right]^{1-\frac{1}{\alpha}}} \lambda^{-\frac{2m}{k}} \lVert w\rVert^2_\Omega \nonumber\\
& + (1-\alpha) |\Omega| L_F^2 \dfrac{\lVert \lambda^{-\frac{m}{k}} \rVert^2_{(t_1,t_2)\times (a_1,a_2)}}{(t_2 - t_1)(a_2 - a_1)}.
\label{eq:3.19}
\end{align}
Notice that for $\alpha \in (0,1]$, it holds $\lambda^{-\frac{2m}{k\alpha} + \frac{2m}{k}} = \lambda^{\frac{2m}{k} \left(1-\frac{1}{\alpha}\right)} \leq \eta^{\frac{2m}{k} \left(1-\frac{1}{\alpha}\right)}$. Using this bound and integrating both sides of \eqref{eq:3.19} over $\left(t_2, T\right) \times \left(a_2, a_\dagger\right)$, we obtain
\begin{align}
& \left\Vert \lambda^{-\frac{m}{k}}\left(\mathcal{A}w-w_{t}-w_{a}\right)\right\Vert _{Q_{\left(t_{2},T\right)\times\left(a_{2},a_{\dagger}\right)}^{\Omega}}^{2} \nonumber \\
& \le\alpha L_{F}^{2}\frac{\left\Vert \lambda^{-\frac{m}{k}}\right\Vert _{\left(t_{1},t_{2}\right)\times\left(a_{1},a_{2}\right)}^{2\left(1-\frac{1}{\alpha}\right)}}{\left(t_2 - t_1\right)^{1-\frac{1}{\alpha}}\left(a_{2}-a_1\right)^{1-\frac{1}{\alpha}}}\left\Vert \lambda^{-\frac{m}{k}}w\right\Vert _{Q_{\left(t_{2},T\right)\times\left(a_{2},a_{\dagger}\right)}^{\Omega}}^{2}\eta^{\frac{2m}{k}\left(1-\frac{1}{\alpha}\right)} \nonumber\\
& +\left(1-\alpha\right)\left|\Omega\right|\frac{\left(T-t_{2}\right)\left(a_{\dagger}-a_{2}\right)}{\left(t_2 - t_1\right)(a_2 - a_1)}L_{F}^{2}\left\Vert \lambda^{-\frac{m}{k}}\right\Vert _{\left(t_{1},t_{2}\right)\times\left(a_{1},a_{2}\right)}^{2}.
\label{eq:moi}
\end{align}
Using that $\left|\lambda\right|\le \mu_{0}'+\eta$, we have
\begin{align}
& \left\Vert \lambda^{-\frac{m}{k}}\right\Vert _{\left(t_{1},t_{2}\right)\times\left(a_{1},a_{2}\right)}^{2\left(1-\frac{1}{\alpha}\right)}\le \left(t_2-t_1\right)^{1-\frac{1}{\alpha}}\left(a_{2}-a_1\right)^{1-\frac{1}{\alpha}}\left(\mu_{0}'+\eta\right)^{\frac{2m}{k}\left(\frac{1}{\alpha}-1\right)}, \label{eq:3.13-3}\\
& \left\Vert \lambda^{-\frac{m}{k}}w\right\Vert _{Q_{\left(t_{2},T\right)\times\left(a_{2},a_{\dagger}\right)}^{\Omega}}^{2}\le\left(\mu_{0}'+\eta\right)^{2}\left\Vert \lambda^{-\frac{m}{k}-1}w\right\Vert _{Q_{\left(t_{2},T\right)\times\left(a_{2},a_{\dagger}\right)}^{\Omega}}^{2}.
\label{eq:3.13}
\end{align}


Plugging \eqref{eq:moi}-\eqref{eq:3.13} into \eqref{eq:3.13-1} and by the choice \eqref{eq:choice}, we obtain
\begin{align}\label{eq:3.17}
K\left\Vert \lambda^{-\frac{m}{k}}\left(\mathcal{A}v-v_{t}-v_{a}\right)\right\Vert _{Q_{\left(t_{1},t_{2}\right)\times\left(a_{1},a_{2}\right)}^{\Omega}}^{2}&  \\  +K\left|\Omega\right|\frac{\left(T-t_{2}\right)\left(a_{\dagger}-a_{2}\right)}{\left(t_2 - t_1\right)(a_2 - a_1)}L_{F}^{2}\left\Vert \lambda^{-\frac{m}{k}}\right\Vert _{\left(t_{1},t_{2}\right)\times\left(a_{1},a_{2}\right)}^{2} 
&\ge \frac{1}{2}\left\Vert \lambda^{-\frac{m}{k}-1}w\right\Vert ^{2}_{Q_{\left(t_{2},T\right)\times\left(a_{2},a_{\dagger}\right)}^{\Omega}}.
\nonumber
\end{align}

For any $t_{2}<t_{3}<T$ and $a_{2}<a_{3}<a_{\dagger}$, \eqref{eq:3.17}
can be further estimated by
\begin{align}
 & K\left(T+a_{\dagger}+\eta-t_{2}-a_{2}\right)^{-\frac{2m}{k}}\left(\left\Vert \mathcal{A}v-v_{t}-v_{a}\right\Vert _{Q_{\left(t_{1},t_{2}\right)\times\left(a_{1},a_{2}\right)}^{\Omega}}^{2}+Ta_{\dagger}\left|\Omega\right|L_{F}^{2}\right) \nonumber \\
& \ge\frac{\left(T+a_{\dagger}+\eta-t_{3}-a_{3}\right)^{-\frac{2m}{k}-2}}{2}\left\Vert w\right\Vert _{Q_{\left(t_{3},T\right)\times\left(a_{3},a_{\dagger}\right)}^{\Omega}}^{2}.\label{eq:3.18-1}
\end{align}

Observing that
\begin{align}\label{limit}
\left(\frac{T+a_{\dagger}+\eta-t_{2}-a_{2}}{T+a_{\dagger}+\eta-t_{3}-a_{3}}\right)^{-\frac{2m}{k}}\to0\quad\text{as }m\to\infty,
\end{align}
we obtain from \eqref{eq:3.18-1} that $w\equiv0$ for $\left(t,a,x\right)\in Q_{\left(t_{3},T\right)\times\left(a_{3},a_{\dagger}\right)}^{\Omega}$
and thus for $\left(t,a,x\right)\in Q_{T,a_{\dagger}}^{\Omega}$ since
$0<t_1 < t_2 < t_{3}$ and $0<a_1 < a_2 < a_3$ can be taken arbitrarily small.

Hence, we state the following uniqueness theorem.
\begin{theorem}\label{thm:main1}
Assume $\left(\text{A}_{1}\right)$-$\left(\text{A}_{3}\right)$ hold.
Then, the problem \eqref{eq:P} with the Dirichlet boundary condition \eqref{Dirichlet} admits no more than one solution in $W_{T,a_{\dagger}}^{\Omega}$.
\end{theorem}

\begin{remark}
The presence of $k>0$ in the context of Lemma \ref{lem:3.1-1} implies that $m$ can be taken as a real positive number. As a corollary, the uniqueness result holds when the source
term $F$ is globally Lipschitz, i.e. $\alpha=1$.
When $F$ is locally Lipschitz, it can also be reduced to the global
case if we set
\[
\widehat{W}_{T,a_{\dagger}}^{\Omega}=C\left(\overline{Q_{T,a_{\dagger}}};H_{0}^{1}\left(\Omega\right)\cap L^{\infty}\left(\Omega\right)\right)\cap L^{2}\left(Q_{T,a_{\dagger}};H^{2}\left(\Omega\right)\right)\cap C^{1}\left(Q_{T,a_{\dagger}};H_{0}^{1}\left(\Omega\right)\right).
\]
This space implies that the population density is essentially bounded in space
and uniformly bounded in time and age. This further gives us an idea to obtain the uniqueness of solution when $\alpha>1$. In fact, if
we assign this boundedness to a constant $r>0$ and apply the elementary
inequality $\left|X\right|^{\alpha}\le\alpha r^{\alpha-1}\left|X\right|$ provided that $\left|X\right|\le r$ and $\alpha>1$, then we get
back the globally Lipschitz case.\\
\indent
Concerning the existence of the function $\kappa$ in \eqref{eq:kappa},
we can rely on the application of partitions of unity (cf. \cite[Proposition 2.25]{Lee03}).
It says that if $M$ is a smooth manifold with or without boundary,
then for any closed subset $A\subseteq M$ and any open subset $U$
containing $A$, there exists a smooth bump function for $A$ supported
in $U$. Accordingly, the existence of $\kappa$ is guaranteed by taking,
\[
A=\left[t_{2},T\right]\times\left[a_{2},a_{\dagger}\right],\;M=\overline{Q_{T,a_{\dagger}}}\;\text{and }U=\left\{ \left(t_{1},t_{2}\right)\times\left(a_{1},a_{2}\right)\right\} \cup A.
\]
\end{remark}

\subsection{Robin-type boundary condition}
Similar to the previous subsection, to prove uniqueness we consider $u_1$ and $u_2$ as two solutions (in some appropriate spaces) of \eqref{eq:P} and \eqref{Robin} and then denote $w = u_1-u_2$, which satisfies

\begin{equation}
\begin{cases}
w_{t}+w_{a}-\mathcal{A}w=F\left(t,a,x;u_1\right)-F\left(t,a,x;u_2\right) & \text{in }Q_{T,a_{\dagger}}^\Omega,\\
-\text{d}\left(t,a,x\right)\nabla w\left(t,a,x\right)\cdot\text{n}=S\left(u_1\right)-S\left(u_2\right) & \text{on }Q_{T,a_{\dagger}}^{\partial\Omega},\\
w\left(T,a,x\right)=0 & \text{in }\left(0,a_{\dagger}\right)\times\Omega,\\
w\left(t,a_{\dagger},x\right)=0 & \text{in }\left(0,T\right)\times\Omega.
\end{cases}\label{eq:Pnew-1}
\end{equation}

We set the following function spaces:
\begin{align}
& \widetilde{W}_{T,a_{\dagger}}^{\Omega}:=C\left(\overline{Q_{T,a_{\dagger}}};H^{1}\left(\Omega\right)\right)\cap L^{2}\left(Q_{T,a_{\dagger}};H^{2}\left(\Omega\right)\right)\cap C^{1}\left(Q_{T,a_{\dagger}};H^{1}\left(\Omega\right)\right), \label{W2}
\\
& \widetilde{P}_{T,a_{\dagger}}^{\Omega}:=\left\{ u\in \widetilde{W}_{T,a_{\dagger}}^{\Omega}:\left.u\right|_{t\in\left\{ 0,T\right\} }=0,\left.u\right|_{a\in\left\{ 0,a_{\dagger}\right\} }=0\right\} .\label{P2}
\end{align}
The space $\widetilde{P}_{T,a_{\dagger}}^{\Omega}$
does not now contain any boundary information on $\partial \Omega$ as in $P_{T,a_{\dagger}}^{\Omega}$ due to the Robin-type boundary condition \eqref{Robin} instead of the Dirichlet boundary condition \eqref{Dirichlet}. Thus, Lemma \ref{lem:3.1-1} cannot be applied
with any function in $\widetilde{P}_{T,a_{\dagger}}^{\Omega}$. However, we are able to formulate the following lemma with $w\in \widetilde{P}_{T,a_{\dagger}}^{\Omega}$ directly as a solution of
\eqref{eq:Pnew-1}.
Here, we also set $\lambda\left(t,a\right):=t-T+a-a_{\dagger}-\eta$
for $\eta>0$.
\begin{lemma}
\label{lem:3.2-1}Let $\beta \in (0,1)$ and assume $\left(\text{A}_{1}\right)$, $\left(\text{A}_{2}\right)$
and $\left(\text{A}_{4}\right)$-$\left(\text{A}_{6}\right)$ hold. Let $w\in \widetilde{P}_{T,a_{\dagger}}^{\Omega}$
be a solution of the problem \eqref{eq:Pnew-1}. Then, for any $m\in \mathbb{N}^{*}$, $k\in \mathbb{R}^{*}_{+}$ and $0<t_1 < t_2 < T$, $0< a_1 < a_2 < a_{\dagger}$, we have
\begin{align}\label{eq:3.23-1}
 & \left\Vert \lambda^{-\frac{m}{k}}\left(\mathcal{A}w-w_{t}-w_{a}\right)\right\Vert _{Q_{T,a_{\dagger}}^{\Omega}}^{2} \\
& \ge\frac{4m}{k}\left\Vert \lambda^{-\frac{m}{k}-1}w\right\Vert _{Q_{T,a_{\dagger}}^{\Omega}}^{2}-\overline{M}\left\Vert \lambda^{-\frac{m}{k}}\nabla w\right\Vert _{Q_{T,a_{\dagger}}^{\Omega}}^{2}-2\overline{m}\left\Vert \lambda^{-\frac{m}{k}}w\right\Vert _{Q_{T,a_{\dagger}}^{\partial\Omega}}^{2} \nonumber \\
& -\dfrac{\eta^2}{4C_0}\left\Vert \lambda^{-\frac{m}{k}-1}w\right\Vert _{Q_{T,a_{\dagger}}^{\partial\Omega}}^{2}
 -16 C_0\textbf{\emph{A}},\nonumber
\end{align}
where 
\begin{align}\label{AAA}
\textbf{\emph{A}} & = \beta L_{S}^{2}\frac{\eta^{\frac{2m}{k}\left(1-\frac{1}{\beta}\right)+2}\left\Vert \lambda^{-\frac{m}{k}}\right\Vert _{\left(t_{1},t_{2}\right)\times\left(a_{1},a_{2}\right)}^{\frac{2}{\beta}\left(\beta-1\right)}}{\left[\left(t_{2}-t_{1}\right)\left(a_{2}-a_{1}\right)\right]^{1-\frac{1}{\beta}}}\left\Vert \lambda^{-\frac{m}{k}}w\right\Vert _{Q_{T,a_{\dagger}}^{\partial\Omega}}^{2} \\
& +\left(\frac{m}{k\eta^{\beta+1}}\right)^{\frac{2}{1-\beta}}\left(1-\beta\right)\left|\partial\Omega\right|L_{S}^{2}\frac{Ta_{\dagger}\left\Vert \lambda^{-\frac{m}{k}}\right\Vert _{\left(t_{1},t_{2}\right)\times\left(a_{1},a_{2}\right)}^{2}}{\left(t_{2}-t_{1}\right)\left(a_{2}-a_{1}\right)}.\nonumber
\end{align}
Also, if $0<T+a_{\dagger}\le\mu$ for sufficiently small 
$\mu>0$, there exists $K>0$ such that
\begin{align}
 K\left\Vert \lambda^{-\frac{m}{k}}\left(\mathcal{A}w-w_{t}-w_{a}\right)\right\Vert _{Q_{T,a_{\dagger}}^{\Omega}}^{2} +K\textbf{\emph{A}} & \ge\left\Vert \lambda^{-\frac{m}{k}-1}w\right\Vert _{Q_{T,a_{\dagger}}^{\Omega}}^{2}+\frac{1}{2}\left\Vert \lambda^{-\frac{m}{k}}\nabla w\right\Vert _{Q_{T,a_{\dagger}}^{\Omega}}^{2},
\label{eq:3.24-1}
\end{align}
for sufficiently large $m$.
\end{lemma}

\begin{proof}
We can adapt the proof of Lemma \ref{lem:3.1-1}
to prove the estimates \eqref{eq:3.23-1} and \eqref{eq:3.24-1}.
Omitting some calculus, 
let $z=\lambda^{-\frac{m}{k}}w\in \widetilde{P}_{T,a_{\dagger}}^{\Omega}$, then (as in \eqref{eq:3.3.1})

\begin{align}
\left\Vert \lambda^{-\frac{m}{k}}\left(\mathcal{A}w-w_{t}-w_{a}\right)\right\Vert _{Q_{T,a_{\dagger}}^{\Omega}}^{2} & =\left\Vert z_{t}+z_{a}\right\Vert _{Q_{T,a_{\dagger}}^{\Omega}}^{2}-2\left\langle z_{t}+z_{a},\mathcal{A}z\right\rangle _{Q_{T,a_{\dagger}}^{\Omega}}\nonumber \\
 & +\dfrac{4m}{k}\left\langle \lambda^{-1}z,z_{t}+z_{a}\right\rangle _{Q_{T,a_{\dagger}}^{\Omega}} +\left\Vert \mathcal{A}z-\dfrac{2m}{k}\lambda^{-1}z\right\Vert _{Q_{T,a_{\dagger}}^{\Omega}}^{2}.\label{eq:3.26-1}
\end{align}
Using the Robin boundary condition
from \eqref{eq:Pnew-1}, we obtain
\begin{align}
-2\left\langle z_{t}+z_{a},\mathcal{A}z\right\rangle _{L^{2}\left(Q_{T,a_{\dagger}}^{\Omega}\right)} & =\underbrace{-2\int_{0}^{T}\int_{0}^{a_{\dagger}}\int_{\partial\Omega}\left(z_{t}+z_{a}\right)\text{d}\nabla z\cdot\text{n}d\sigma_{x}dadt}_{:=\mathcal{I}_{1}}\nonumber \\
 & +\underbrace{2\int_{0}^{T}\int_{0}^{a_{\dagger}}\int_{\Omega}\sum_{1\leq i,j\leq n}\left(\partial_{t}\partial_{x_{i}}z+\partial_{a}\partial_{x_{i}}z\right)d_{ij}\partial_{x_{j}}zdxdadt}_{:=\mathcal{I}_{2}}.\label{eq:3.27-1}
\end{align}
As in the derivation of \eqref{eq:3.3.2} and using $\left(\text{A}_{2}\right)$ we obtain
\begin{eqnarray}
\mathcal{I}_{2} =-\int_{0}^{T}\int_{0}^{a_{\dagger}}\int_{\Omega}\sum_{1\leq i,j\leq n}\partial_{x_{i}}z\left(\partial_{t}+\partial_{a}\right)\left(d_{ij}\right)
\partial_{x_{j}}zdxdadt \nonumber \\
\ge-\overline{M}\left\Vert \nabla z\right\Vert _{L^{2}\left(Q_{T,a_{\dagger}}^{\Omega}\right)}^{2}.\label{3.27-2}
\end{eqnarray}
To estimate $\mathcal{I}_{1}$, recall the Robin boundary
condition from \eqref{eq:Pnew-1} to get
\begin{align}
\mathcal{I}_{1} & =2\int_{0}^{T}\int_{0}^{a_{\dagger}}\int_{\partial\Omega}\left(z_{t}+z_{a}\right)\lambda^{-\frac{m}{k}}\left(S\left(u_1\right)-S\left(u_2\right)\right)d\sigma_{x}dadt\nonumber \\
 & =\underbrace{2\int_{0}^{T}\int_{0}^{a_{\dagger}}\int_{\partial\Omega}\left(\partial_{t}+\partial_{a}\right)\left[\lambda^{-\frac{m}{k}}z\left(S\left(u_1\right)-S\left(u_2\right)\right)\right]d\sigma_{x}dadt}_{:=\mathcal{I}_{3}}\nonumber \\
 & \underbrace{-2\int_{0}^{T}\int_{0}^{a_{\dagger}}\int_{\partial\Omega}\lambda^{-\frac{m}{k}}z\left(\partial_{t}+\partial_{a}\right)\left(S\left(u_1\right)-S\left(u_2\right)\right)d\sigma_{x}dadt}_{:=\mathcal{I}_{4}}\nonumber \\
 & \underbrace{+\frac{4m}{k}\int_{0}^{T}\int_{0}^{a_{\dagger}}\int_{\partial\Omega}\lambda^{-\frac{m}{k}-1}z\left(S\left(u_1\right)-S\left(u_2\right)\right)d\sigma_{x}dadt}_{:=\mathcal{I}_{5}}.\label{eq:3.28-1}
\end{align}

Notice that due to the zero conditions in the definition \eqref{P2} of $\widetilde{P}_{T,a_{\dagger}}^{\Omega}$,
the first term $\mathcal{I}_{3}$ of \eqref{eq:3.28-1} vanishes. Using $\left(\text{A}_{4}\right)$, we estimate
$\mathcal{I}_{4}$ by
\begin{align}
\mathcal{I}_{4} & =-2\int_{0}^{T}\int_{0}^{a_{\dagger}}\int_{\partial\Omega}\lambda^{-\frac{2m}{k}}\left(u_1-u_2\right)\left(\partial_{t}+\partial_{a}\right)\left(S\left(u_1\right)-S\left(u_2\right)\right)d\sigma_{x}dadt \nonumber \\
 & \ge-2\overline{m}\left\Vert \lambda^{-\frac{m}{k}}w\right\Vert _{Q_{T,a_{\dagger}}^{\partial\Omega}}^{2}.
\label{bdtI4}
\end{align}
By the back-substitution
$z=\lambda^{-\frac{m}{k}}w$, the term $\mathcal{I}_{5}$ can be estimated by

\begin{align}
\mathcal{I}_{5} \ge-\frac{2m}{k}\left(\frac{k\eta^2 }{8mC_0}\left\Vert \lambda^{-\frac{m}{k}-1}w\right\Vert _{Q_{T,a_{\dagger}}^{\partial\Omega}}^{2}
+ \frac{8mC_0}{k\eta^2 }\left\Vert \lambda^{-\frac{m}{k}}\left(S\left(u_1\right)-S\left(u_2\right)\right)\right\Vert _{Q_{T,a_{\dagger}}^{\partial\Omega}}^{2}\right).
\label{eq:bdt_{b}}
\end{align}
Apply $\left(\text{A}_{6}\right)$ and the inequality \eqref{eq:bdtphu} for $X=\left|w\right|^{2}\geq 0$, $\alpha_0 = \beta \in (0,1)$ and
\[
\gamma=\left[\left(t_{2}-t_{1}\right)\left(a_{2}-a_{1}\right)\right]^{-\frac{1}{\beta}}\left(\frac{\eta^2 k}{m}\right)^{\frac{2}{\beta - 1}}\lambda^{\frac{2m}{k\beta}}\left\Vert \lambda^{-\frac{m}{k}}\right\Vert _{\left(t_{1},t_{2}\right)\times\left(a_{1},a_{2}\right)}^{\frac{2}{\beta}}>0,
\]
for $0<t_1 < t_2 < T, 0< a_1 < a_2 < a_{\dagger}$, to get (using that $\lvert \lambda\rvert > \eta$), as in \eqref{eq:3.19},
\begin{align}
\frac{m^2 }{k^2\eta^2 }\left\Vert \lambda^{-\frac{m}{k}}\left(S\left(u_1\right)-S\left(u_2\right)\right)\right\Vert _{Q_{T,a_{\dagger}}^{\partial\Omega}}^{2} &\le \textbf{A} . 
\label{eq:bdt_b}
\end{align}

As in \eqref{eq:3.3.3} since $z\in \widetilde{P}_{T,a_{\dagger}}^{\Omega}$, we observe that
\begin{equation}
\dfrac{4m}{k}\left\langle \lambda^{-1}z,z_{t}+z_{a}\right\rangle _{Q_{T,a_{\dagger}}^{\Omega}} = \dfrac{4m}{k}\left\Vert \lambda^{-1}z\right\Vert _{Q_{T,a_{\dagger}}^{\Omega}}^{2}.\label{eq:3.32}
\end{equation}

We complete the proof of \eqref{eq:3.23-1} by grouping together \eqref{eq:3.26-1}, \eqref{3.27-2}, \eqref{bdtI4}-\eqref{eq:3.32}. 

It remains to prove the estimate \eqref{eq:3.24-1}. 
Using Green's formula we have  
\begin{align}\label{BBB}
 -\left\langle \lambda^{-\frac{2m}{k}}w,\mathcal{A}w-w_{t}-w_{a}\right\rangle _{Q_{T,a_{\dagger}}^{\Omega}}  &=\underbrace{\int_{0}^{T}\int_{0}^{a_{\dagger}}\int_{\partial\Omega}\lambda^{-\frac{2m}{k}}w\left(S\left(u_1\right)-S\left(u_2\right)\right)d\sigma_{x}dadt}_{:=\mathcal{I}_{6}}\nonumber \\  & 
 +\underbrace{\int_{0}^{T}\int_{0}^{a_{\dagger}}\int_{\Omega}\lambda^{-\frac{2m}{k}} \left( \text{d}\nabla w\right)\cdot \nabla w dxdadt}_{:=\mathcal{I}_{7}}\nonumber
 \\& +\underbrace{\int_{0}^{T}\int_{0}^{a_{\dagger}}\int_{\Omega}\lambda^{-\frac{2m}{k}}w\left(w_{t}+w_{a}\right)dxdadt}_{:=\mathcal{I}_{8}}.
\end{align}

Since $\left|\lambda\left(t,a\right)\right|\le T+a_{\dagger}+\eta$
for all $\left(t,a\right)\in\overline{Q_{T,a_{\dagger}}}$, $\mathcal{I}_{8}$ can be estimated using 
integration by parts, while for $\mathcal{I}_{6}$ and $\mathcal{I}_{7}$
we use $\left(\text{A}_{1}\right)$ and
$\left(\text{A}_{5}\right)$ to obtain
\begin{align}
-\left\langle \lambda^{-\frac{2m}{k}}w,\mathcal{A}w-w_{t}-w_{a}\right\rangle _{Q_{T,a_{\dagger}}^{\Omega}} &\ge\underline{c}\left\Vert \lambda^{-\frac{m}{k}}\nabla w\right\Vert _{Q_{T,a_{\dagger}}^{\Omega}}^{2}\nonumber\\
&-\frac{2m}{k}\left(T+a_{\dagger}+\eta\right)\left\Vert \lambda^{-\frac{m}{k}-1}w\right\Vert _{Q_{T,a_{\dagger}}^{\Omega}}^{2}.\label{eq:3.28}
\end{align}

Using Young's inequality
\begin{align*}
& \dfrac{1}{2} \left\Vert \lambda^{-\frac{m}{k}-1}w\right\Vert _{Q_{T,a_{\dagger}}^{\Omega}}^{2} + \dfrac{1}{2} \left\Vert \lambda^{-\frac{m}{k}+1}\left(\mathcal{A}w-w_{t}-w_{a}\right)\right\Vert _{Q_{T,a_{\dagger}}^{\Omega}}^{2} \\
 & \ge -\left\langle \lambda^{-\frac{2m}{k}}w,\mathcal{A}w-w_{t}-w_{a}\right\rangle _{Q_{T,a_{\dagger}}^{\Omega}}, 
\end{align*}
then \eqref{eq:3.28} yields (using also that $\left| \lambda \right| \leq T + a_\dagger +\eta$)

\begin{align}\label{eq:3.28-2-1}
& K_2 \left\Vert \lambda^{-\frac{m}{k}-1}w\right\Vert _{Q_{T,a_{\dagger}}^{\Omega}}^{2} + 
\dfrac{\left(T + a_\dagger + \eta\right)^2}{2} \left\Vert \lambda^{-\frac{m}{k}}\left(\mathcal{A}w-w_{t}-w_{a}\right)\right\Vert _{Q_{T,a_{\dagger}}^{\Omega}}^{2}  \\
& \ge \underline{c}\left\Vert \lambda^{-\frac{m}{k}}\nabla w\right\Vert _{Q_{T,a_{\dagger}}^{\Omega}}^{2} \nonumber
\end{align}
where $K_2 := \frac{1}{2} + \frac{2m}{k} \left(T+a_\dagger +\eta\right)$. Now using \eqref{eq:bdt_trace} in the third and fourth terms of the RHS of \eqref{eq:3.23-1}, we obtain (using also that $\left| \lambda\right| >\eta$)
\begin{align}\label{eq:3.28-3}
 & \left\Vert \lambda^{-\frac{m}{k}}\left(\mathcal{A}w-w_{t}-w_{a}\right)\right\Vert _{Q_{T,a_{\dagger}}^{\Omega}}^{2} \\
& \ge K_1\left\Vert \lambda^{-\frac{m}{k}-1}w\right\Vert _{Q_{T,a_{\dagger}}^{\Omega}}^{2}-\left(\overline{M} +2\overline{m}C_0+\frac{1}{4}\right)\left\Vert \lambda^{-\frac{m}{k}}\nabla w\right\Vert _{Q_{T,a_{\dagger}}^{\Omega}}^{2} -16 C_0\textbf{A},\nonumber
\end{align}
where $K_1 :=\frac{4m}{k} -2\overline{m}C_0 \left(T+a_\dagger +\eta\right)^2 -\frac{\eta^2}{4}$. Multiplying \eqref{eq:3.28-3} by $K_2 K^{-1}_1$ and applying \eqref{eq:3.28-2-1}, we obtain

\begin{align}
\left[\underline{c}-K_2 K_1^{-1} \left(\overline{M}+2C_0 \overline{m} +\frac{1}{4}\right)\right]\left\Vert \lambda^{-\frac{m}{k}}\nabla w\right\Vert _{Q_{T,a_{\dagger}}^{\Omega}}^{2}\le \mathcal{I}_{9}+\mathcal{I}_{10},
\label{eq:3.30}
\end{align}
which in line with the estimate \eqref{eq:3.10}. In \eqref{eq:3.30}, we have denoted 
\begin{align*}
\mathcal{I}_{9}&:=\left[K_2 K_{1}^{-1} + \dfrac{\left(T+a_{\dagger}+\eta\right)^2}{2} \right]\left\Vert \lambda^{-\frac{m}{k}}\left(\mathcal{A}w-w_{t}-w_{a}\right)\right\Vert _{Q_{T,a_{\dagger}}^{\Omega}}^{2},\\
\mathcal{I}_{10}&:=16 K_2 K_{1}^{-1}C_{0}\textbf{A}.
\end{align*}

Since $K_2 K_{1}^{-1} \to \frac{1}{2}(T + a_{\dagger} + \eta)$ as $m \to \infty$, if $0<T+a_\dagger \le \mu$ for sufficiently small $\mu >0$, then can choose $\mu_0 \ge \mu \ge T + a_{\dagger}>0$, $\eta_0\ge \eta > 0$ and $m_0 \le  m$ such that $K_1 > 0$ and $2K_{2}K_{1}^{-1}\left(\overline{M} + 2 C_0 \overline{m} +\frac{1}{4}\right) \le \underline{c}$. Then, for any $m\ge m_0$ denoting
\[
C_{3}:= \frac{2}{\underline{c}} \max\left\{K_2 K_{1}^{-1} + \dfrac{\left(\mu +\eta\right)^2}{2},16 K_2 K_{1}^{-1}C_{0}\right\},
\]
from \eqref{eq:3.28-3} and \eqref{eq:3.30} we obtain
\begin{align}\label{eq:3.31-1}
	& \left\Vert \lambda^{-\frac{m}{k}-1}w\right\Vert _{Q_{T,a_{\dagger}}^{\Omega}}^{2}
	+\left[1 - K_{1}^{-1}\left(\overline{M}+2C_{0}\overline{m} + \frac{1}{4}\right)\right]\left\Vert \lambda^{-\frac{m}{k}}\nabla w\right\Vert _{Q_{T,a_{\dagger}}^{\Omega}}^{2} \\
	& \le\left(K_{1}^{-1}+C_{3}\right)\left(\left\Vert \lambda^{-\frac{m}{k}}\left(\mathcal{A}w-w_{t}-w_{a}\right)\right\Vert _{Q_{T,a_{\dagger}}^{\Omega}}^{2} +\textbf{A}\right) \nonumber
\end{align}

We can now take $m_0$ sufficiently large, and $\mu_0$ and $\eta_0$ sufficiently small such that $K_{1}^{-1}\left(\overline{M}+2C_{0}\overline{m} + \frac{1}{4}\right) \le \frac{1}{2}$ for any $m \ge m_0$. Then, we conclude that \eqref{eq:3.24-1} holds for $m\ge m_0$ by choosing $K := K_3^{-1} + C_4$, 
where
\begin{align*}
&K_3 := \frac{4m_0}{k} - \frac{\eta^2}{4} - 2\overline{m} (\mu_0 + \eta_0)^2 C_0,\quad K_4 := \frac{1}{2} + \frac{2m_0}{k} \left(\mu_0 +\eta_0\right),\\
&C_{4}:= \frac{2}{\underline{c}} \max\left\{K_4 K_{3}^{-1} + \dfrac{\left(\mu_0 +\eta_0\right)^2}{2},16 K_4 K_{3}^{-1}C_{0}\right\},
\end{align*}
using that $K_4 K_3^{-1} \ge K_2 K_1^{-1}$ since the function $K_2 K_1^{-1}$ is monotonically decreasing, as a function of $m$.
\end{proof}
Remark that in deriving \eqref{eq:3.23-1} and \eqref{eq:3.24-1}, we have never used that $w$ satisfies the first equation in \eqref{eq:Pnew-1}. Let us now define
\[
\overline{\kappa}\left(t,a\right)=\begin{cases}
0 & \text{if }\left(t,a\right)\in Q_{T,a_{\dagger}}\backslash\left(\left\{ \left(t_{1},t_{2}\right)\times\left(a_{1},a_{2}\right)\right\} \cup\left[t_{2},T\right]\times\left[a_{2},a_{\dagger}\right]\right),\\
1 & \text{if }\left(t,a\right)\in\left\{ \left(t_{1},t_{2}\right)\times\left(a_{1},a_{2}\right)\right\} \cup\left[t_{2},T\right]\times\left[a_{2},a_{\dagger}\right],
\end{cases}
\]
and set $v:=\overline{\kappa}w\in \widetilde{P}_{T,a_{\dagger}}^{\Omega}$. Then, from the second equation of \eqref{eq:Pnew-1} we have
\begin{align*}
	& -\text{d}\left(t,a,x\right)\nabla v\left(t,a,x\right)\cdot\text{n}\\
	& =\begin{cases}
		S\left(u_{1}\right)-S\left(u_{2}\right)\text{ on }\left(\left(t_{1},t_{2}\right)\times\left(a_{1},a_{2}\right)\cup\left[t_{2},T\right]\times\left[a_{2},a_{\dagger}\right]\right)\times\partial\Omega,\\
		0\text{ on }\left(Q_{T,a_{\dagger}}\backslash\left(\left(t_{1},t_{2}\right)\times\left(a_{1},a_{2}\right)\cup\left[t_{2},T\right]\times\left[a_{2},a_{\dagger}\right]\right)\right)\times\partial\Omega.
	\end{cases}
\end{align*}
Then, it can be remarked that the inequality \eqref{eq:3.24-1} in Lemma \ref{lem:3.2-1} holds not only for $w$ but also for $v=\overline{\kappa}w$ since in \eqref{BBB} we only need to use that $\mathcal{I}_{6}\ge 0$, with the rest of the argument remaining the same. Then, from \eqref{eq:3.24-1} applied to $v$ we have (since $v=w$ in $\left(\left(t_{1},t_{2}\right)\times\left(a_{1},a_{2}\right)\cup\left[t_{2},T\right]\times\left[a_{2},a_{\dagger}\right]\right)\times\Omega$ and $v=0$ otherwise)
\begin{align}\label{eq:33}
	&K\left\Vert \lambda^{-\frac{m}{k}}\left(\mathcal{A}w-w_{t}-w_{a}\right)\right\Vert _{Q_{\left(t_{1},t_{2}\right)\times\left(a_{1},a_{2}\right)}^{\Omega}}^{2}  +K\left\Vert \lambda^{-\frac{m}{k}}\left(\mathcal{A}w-w_{t}-w_{a}\right)\right\Vert _{Q_{\left(t_{2},T\right)\times\left(a_{2},a_{\dagger}\right)}^{\Omega}}^{2} \\&
	+K\mathbf{A}_{1}+K\mathbf{A}_{2}\ge\left\Vert \lambda^{-\frac{m}{k}-1}w\right\Vert _{Q_{\left(t_{2},T\right)\times\left(a_{2},a_{\dagger}\right)}^{\Omega}}^{2}  +\frac{1}{2}\left\Vert \lambda^{-\frac{m}{k}}\nabla w\right\Vert _{Q_{\left(t_{2},T\right)\times\left(a_{2},a_{\dagger}\right)}^{\Omega}}^{2},\nonumber
\end{align}
where
\begin{align*}
	\mathbf{A}_{1} & =\beta L_{S}^{2}\frac{\eta^{\frac{2m}{k}\left(1-\frac{1}{\beta}\right)+2}\left\Vert \lambda^{-\frac{m}{k}}\right\Vert _{\left(t_{1},t_{2}\right)\times\left(a_{1},a_{2}\right)}^{\frac{2}{\beta}\left(\beta-1\right)}}{\left[\left(t_{2}-t_{1}\right)\left(a_{2}-a_{1}\right)\right]^{1-\frac{1}{\beta}}}\left\Vert \lambda^{-\frac{m}{k}}w\right\Vert _{Q_{\left(t_{1},t_{2}\right)\times\left(a_{1},a_{2}\right)}^{\partial\Omega}}^{2}\\
	& +\left(\frac{m}{k\eta^{\beta+1}}\right)^{\frac{2}{1-\beta}}\left(1-\beta\right)\left|\partial\Omega\right|L_{S}^{2}\frac{Ta_{\dagger}\left\Vert \lambda^{-\frac{m}{k}}\right\Vert _{\left(t_{1},t_{2}\right)\times\left(a_{1},a_{2}\right)}^{2}}{\left(t_{2}-t_{1}\right)\left(a_{2}-a_{1}\right)},
\end{align*}
\begin{align*}
	\mathbf{A}_{2} & =\beta L_{S}^{2}\frac{\eta^{\frac{2m}{k}\left(1-\frac{1}{\beta}\right)+2}\left\Vert \lambda^{-\frac{m}{k}}\right\Vert _{\left(t_{1},t_{2}\right)\times\left(a_{1},a_{2}\right)}^{\frac{2}{\beta}\left(\beta-1\right)}}{\left[\left(t_{2}-t_{1}\right)\left(a_{2}-a_{1}\right)\right]^{1-\frac{1}{\beta}}}\left\Vert \lambda^{-\frac{m}{k}}w\right\Vert _{Q_{\left(t_{2},T\right)\times\left(a_{2},a_{\dagger}\right)}^{\partial\Omega}}^{2}.
\end{align*}
Since $ \left|\lambda\right| \le T + a_{\dagger} + \eta\le \mu_{0}' + \eta$, we have that
\begin{align}\label{eq:66}
	\mathbf{A}_{1} & \le\beta L_{S}^{2}\left(\mu_{0}'+\eta\right)^{\frac{2m}{k}\left(\frac{1}{\beta}-1\right)+2}\eta^{\frac{2m}{k}\left(1-\frac{1}{\beta}\right)+2}\left\Vert \lambda^{-\frac{m}{k}-1}w\right\Vert _{Q_{\left(t_{1},t_{2}\right)\times\left(a_{1},a_{2}\right)}^{\partial\Omega}}^{2} \\
	& +\left(\frac{m}{k\eta^{\beta +1}}\right)^{\frac{2}{1-\beta}}\left(1-\beta\right)\left|\partial\Omega\right|L_{S}^{2}\frac{Ta_{\dagger}\left\Vert \lambda^{-\frac{m}{k}}\right\Vert _{\left(t_{1},t_{2}\right)\times\left(a_{1},a_{2}\right)}^{2}}{\left(t_{2}-t_{1}\right)\left(a_{2}-a_{1}\right)}.\nonumber
\end{align}
Also, as in \eqref{eq:3.13-3}, with $\alpha$ replaced by $\beta$, using \eqref{eq:bdt_trace} and $\left|\lambda\right|\le \mu_{0}' + \eta$, we have
\begin{align}\label{eq:44}
	\mathbf{A}_{2}
	& \le\beta L_{S}^{2}\left(\mu_{0}'+\eta\right)^{\frac{2m}{k}\left(\frac{1}{\beta}-1\right)+2}\eta^{\frac{2m}{k}\left(1-\frac{1}{\beta}\right)+2}C_{0}\left\Vert \lambda^{-\frac{m}{k}-1}w\right\Vert _{Q_{\left(t_{2},T\right)\times\left(a_{2},a_{\dagger}\right)}^{\Omega}}^{2} \\
	& +\beta L_{S}^{2}\left(\mu_{0}'+\eta\right)^{\frac{2m}{k}\left(\frac{1}{\beta}-1\right)}\eta^{\frac{2m}{k}\left(1-\frac{1}{\beta}\right)+2}C_0 \left\Vert \lambda^{-\frac{m}{k}}\nabla w\right\Vert _{Q_{\left(t_{2},T\right)\times\left(a_{2},a_{\dagger}\right)}^{\Omega}}^{2}.\nonumber
\end{align}
From \eqref{eq:moi}-\eqref{eq:3.13}, we have that

\begin{align}
& \left\Vert \lambda^{-\frac{m}{k}}\left(\mathcal{A}w-w_{t}-w_{a}\right)\right\Vert _{Q_{\left(t_{2},T\right)\times\left(a_{2},a_{\dagger}\right)}^{\Omega}}^{2}\label{eq:55}\\
& \le\alpha L_{F}^{2}\frac{\eta^{\frac{2m}{k}\left(1-\frac{1}{\alpha}\right)}\left\Vert \lambda^{-\frac{m}{k}}\right\Vert _{\left(t_{1},t_{2}\right)\times\left(a_{1},a_{2}\right)}^{\frac{2}{\alpha}\left(\alpha-1\right)}}{\left[\left(t_{2}-t_{1}\right)\left(a_{2}-a_{1}\right)\right]^{1-\frac{1}{\alpha}}}\left\Vert \lambda^{-\frac{m}{k}}w\right\Vert _{Q_{\left(t_{2},T\right)\times\left(a_{2},a_{\dagger}\right)}^{\Omega}}^{2}\nonumber \\
& +\left(1-\alpha\right)\left|\Omega\right|L_{F}^{2}\frac{Ta_{\dagger}\left\Vert \lambda^{-\frac{m}{k}}\right\Vert _{\left(t_{1},t_{2}\right)\times\left(a_{1},a_{2}\right)}^{2}}{\left(t_{2}-t_{1}\right)\left(a_{2}-a_{1}\right)}\nonumber \\
& \le\alpha L_{F}^{2}\left(\mu_{0}'+\eta\right)^{\frac{2m}{k}\left(\frac{1}{\alpha}-1\right)+2}\eta^{\frac{2m}{k}\left(1-\frac{1}{\alpha}\right)}\left\Vert \lambda^{-\frac{m}{k}-1}w\right\Vert _{Q_{\left(t_{2},T\right)\times\left(a_{2},a_{\dagger}\right)}^{\Omega}}^{2}\nonumber \\
& +\left(1-\alpha\right)\left|\Omega\right|L_{F}^{2}\frac{Ta_{\dagger}\left\Vert \lambda^{-\frac{m}{k}}\right\Vert _{\left(t_{1},t_{2}\right)\times\left(a_{1},a_{2}\right)}^{2}}{\left(t_{2}-t_{1}\right)\left(a_{2}-a_{1}\right)}.\nonumber 
\end{align}

Combining \eqref{eq:66}-\eqref{eq:44}, for sufficiently large $m\ge m_0$ it follows from \eqref{eq:33} (noticing also $\eta^2 \le (\mu_{0}' + \eta)^2$) that
\begin{align}	\label{eq:3.24-2}
	& K\left\Vert \lambda^{-\frac{m}{k}}\left(\mathcal{A}w-w_{t}-w_{a}\right)\right\Vert _{Q_{\left(t_{1},t_{2}\right)\times\left(a_{1},a_{2}\right)}^{\Omega}}^{2} \\
	& +K \beta L_{S}^{2}\left(\mu_{0}'+\eta\right)^{\frac{2m}{k}\left(\frac{1}{\beta}-1\right)+2}\eta^{\frac{2m}{k}\left(1-\frac{1}{\beta}\right)}\Bigg(\left\Vert \lambda^{-\frac{m}{k}-1}w\right\Vert _{Q_{\left(t_{1},t_{2}\right)\times\left(a_{1},a_{2}\right)}^{\partial\Omega}}^{2} \nonumber\\
	&\left.+C_0 \left\Vert \lambda^{-\frac{m}{k}}\nabla w\right\Vert _{Q_{\left(t_{2},T\right)\times\left(a_{2},a_{\dagger}\right)}^{\Omega}}^{2}\right)\nonumber \\
	& +K\left[\alpha L_{F}^{2}\left(\mu_{0}'+\eta\right)^{\frac{2m}{k}\left(\frac{1}{\alpha}-1\right)+2}\eta^{\frac{2m}{k}\left(1-\frac{1}{\alpha}\right)}\right.\nonumber \\
	& \left.+C_{0}\beta L_{S}^{2}\left(\mu_{0}'+\eta\right)^{\frac{2m}{k}\left(\frac{1}{\beta}-1\right)+2}\eta^{\frac{2m}{k}\left(1-\frac{1}{\beta}\right)}\right]\left\Vert \lambda^{-\frac{m}{k}-1}w\right\Vert _{Q_{\left(t_{2},T\right)\times\left(a_{2},a_{\dagger}\right)}^{\Omega}}^{2}\nonumber \\
	& +K\frac{Ta_{\dagger}}{\left(t_{2}-t_{1}\right)\left(a_{2}-a_{1}\right)}\Bigg[\left(1-\alpha\right)\left|\Omega\right|L_{F}^{2}\nonumber \\
	& \left.+\left(1-\beta\right)\left|\partial\Omega\right|L_{S}^{2}\left(\frac{m}{k\eta^{\beta+1}}\right)^{\frac{2}{1-\beta}}\right]\left\Vert \lambda^{-\frac{m}{k}}\right\Vert _{\left(t_{1},t_{2}\right)\times\left(a_{1},a_{2}\right)}^{2}\nonumber \\
	& \ge\left\Vert \lambda^{-\frac{m}{k}-1}w\right\Vert _{Q_{\left(t_{2},T\right)\times\left(a_{2},a_{\dagger}\right)}^{\Omega}}^{2}+\frac{1}{2}\left\Vert \lambda^{-\frac{m}{k}}\nabla w\right\Vert _{Q_{\left(t_{2},T\right)\times\left(a_{2},a_{\dagger}\right)}^{\Omega}}^{2}.\nonumber
\end{align}

Let $0< \eta_{1} \le \min\left\{1,\eta_0\right\}$, take $\eta = \eta_1$ and
\begin{align}
\mu_{0}'=\begin{cases}
\eta^{\frac{m\left(\frac{1}{\beta}-1\right)}{m\left(\frac{1}{\beta}-1\right)+k}}2^{\frac{k}{2m\left(\frac{1}{\beta}-1\right)+2k}}-\eta & \text{if }\alpha\ge\beta,\\
\eta^{\frac{m\left(\frac{1}{\alpha}-1\right)}{m\left(\frac{1}{\alpha}-1\right)+k}}2^{\frac{k}{2m\left(\frac{1}{\alpha}-1\right)+2k}}-\eta & \text{if }\alpha<\beta.
\end{cases}
\label{newchoice1}
\end{align}

Also, for sufficiently large $m_0$, and $\mu_0,\eta_0$ sufficiently small, choose
\begin{align}
0<K\le\frac{1}{8}\min\left\{ \frac{1}{\alpha L_{F}^{2}},\frac{1}{C_{0}\beta L_{S}^{2}},\frac{1}{\beta L_{S}^{2}}\right\} .
\label{newchoice}
\end{align}
With the choice \eqref{newchoice1}, we have that $\left(\mu_{0}'+\eta\right)^{\frac{2m}{k}\left(\frac{1}{\beta}-1\right)+2}\eta^{\frac{2m}{k}\left(1-\frac{1}{\beta}\right)}=2$. Further, with the choice \eqref{newchoice} we have that $2K\left(\alpha L_{F}^{2}+C_{0}\beta L_{S}^{2}\right)\le\frac{1}{2}$ and then \eqref{eq:3.24-2} yields
\begin{align}\label{lastone}
	& K\left\Vert \lambda^{-\frac{m}{k}}\left(\mathcal{A}w-w_{t}-w_{a}\right)\right\Vert _{Q_{\left(t_{1},t_{2}\right)\times\left(a_{1},a_{2}\right)}^{\Omega}}^{2}+\frac{1}{4}\left\Vert \lambda^{-\frac{m}{k}-1}w\right\Vert _{Q_{\left(t_{1},t_{2}\right)\times\left(a_{1},a_{2}\right)}^{\partial\Omega}}^{2}  \nonumber\\
	& +\frac{KTa_{\dagger}}{\left(t_{2}-t_{1}\right)\left(a_{2}-a_{1}\right)}\left[\left|\Omega\right|L_{F}^{2}+\left|\partial\Omega\right|L_{S}^{2}\left(\frac{m}{k\eta^{\beta+1}}\right)^{\frac{2}{1-\beta}}\right]\left\Vert \lambda^{-\frac{m}{k}}\right\Vert _{\left(t_{1},t_{2}\right)\times\left(a_{1},a_{2}\right)}^{2} \nonumber \\
	& \ge\frac{1}{2}\left\Vert \lambda^{-\frac{m}{k}-1}w\right\Vert _{Q_{\left(t_{2},T\right)\times\left(a_{2},a_{\dagger}\right)}^{\Omega}}^{2}+\frac{1}{4}\left\Vert \lambda^{-\frac{m}{k}}\nabla w\right\Vert _{Q_{\left(t_{2},T\right)\times\left(a_{2},a_{\dagger}\right)}^{\Omega}}^{2}.\nonumber
\end{align}

Observe that to complete the uniqueness result, we only need to mimic the way \eqref{eq:3.18-1} was obtained from \eqref{eq:3.17}, and the limit
\[
\left(\frac{m}{k}\right)^{\frac{2}{1-\beta}}\left(\frac{T+a_{\dagger}+\eta-t_{2}-a_{2}}{T+a_{\dagger}+\eta-t_{3}-a_{3}}\right)^{-\frac{2m}{k}}\to0\quad\text{as }m\to\infty
\]
for any $t_2 < t_3 < T$ and $a_2 < a_3 < a_{\dagger}$, and thus  $w\equiv 0$ for $(t,a,x)\in Q_{T,a_{\dagger}}^{\Omega}$.

In case $\beta = 1$, $\mathcal{I}_{5}$ in \eqref{eq:3.28-1} can be directly estimated from \eqref{eq:bdt_{b}}, using the Lipschitz continuity in $\left(\text{A}_{6}\right)$, without the need of employing the inequality \eqref{eq:bdtphu} to obtain \eqref{eq:bdt_b}. The rest of details in obtaining $w \equiv 0$ in case $\beta = 1$ are skipped.

In conclusion, we can state the following uniqueness theorem.

\begin{theorem} \label{thm:main2}
	Assume $\left(\text{A}_{1}\right)$-$\left(\text{A}_{6}\right)$ hold.
	Then, the problem \eqref{eq:P} satisfying the Robin boundary condition \eqref{Robin} admits no more than one solution in $\widetilde{W}_{T,a_{\dagger}}^{\Omega}$.
\end{theorem}

\section*{Acknowledgments}
This work is in commemoration of the first death anniversary of V.A.K{'}s father. V.A.K thanks Prof. Nguyen Huy Tuan for introducing him the ultra-parabolic problem.

\bibliographystyle{plain}
\bibliography{mybib}

\begin{thebibliography}{10}

\bibitem{Av83}
A.~Friedman.
\newblock {\em Partial Differential Equations of Parabolic Type}.
\newblock R.E. Krieger Publishing Company, 1983.

\bibitem{GLW07}
S.A. Gourley, R.~Liu, and J.~Wu.
\newblock Some vector borne diseases with structured host populations:
  extinction and spatial spread.
\newblock {\em SIAM Journal on Applied Mathematics}, 67(2):408--433, 2007.

\bibitem{Hor79}
C.O. Horgan.
\newblock Eigenvalue estimates and the trace theorem.
\newblock {\em Journal of Mathematical Analysis and Applications}, 69:231--242,
  1979.

\bibitem{JM07}
F.C. Jiang and F.W. Meng.
\newblock Explicit bounds on some new nonlinear integral inequality with delay.
\newblock {\em Journal of Computational and Applied Mathematics}, 205:479--486,
  2007.

\bibitem{Lee03}
J.M. Lee.
\newblock {\em Introduction to {S}mooth {M}anifolds}.
\newblock Graduate Texts in Mathematics, Volume 218. Springer-Verlag, New York,
  2003.

\bibitem{Lorenzi1}
L.~Lorenzi.
\newblock An abstract ultraparabolic integrodifferential equation.
\newblock {\em Matematiche}, 58(2):401--435, 1998.

\bibitem{Lorenzi2}
L.~Lorenzi.
\newblock An identification problem for an ultraparabolic integrodifferential
  equation.
\newblock {\em Journal of Mathematical Analysis and Applications},
  234:417--456, 1999.

\bibitem{MR08}
P.~Magal and S.~Ruan.
\newblock {\em Structured Population Models in Biology and Epidemiology}.
\newblock Mathematical {B}iosciences {S}ubseries, Volume 1936. Springer-Verlag,
  Berlin, 2008.

\bibitem{Marco03}
M.D. Marcozzi.
\newblock On the valuation of {A}sian options by variational methods.
\newblock {\em SIAM Journal on Scientific Computing}, 24:1124--1140, 2003.

\bibitem{Mar09}
M.D. Marcozzi.
\newblock Extrapolation discontinuous {G}alerkin method for ultraparabolic
  equations.
\newblock {\em Journal of Computational and Applied Mathematics}, 224:679--687,
  2009.

\bibitem{Mc26}
A.G. McKendrick.
\newblock Applications of mathematics to medical problems.
\newblock {\em Proceedings of the Edinburgh Mathematical Society}, 44:98--130,
  1926.

\bibitem{SZ01}
J.W.H. So, J.~Wu, and X.~Zou.
\newblock A reaction diffusion model for a single species with age structure,
  {I}. {T}ravelling wave fronts on unbounded domains.
\newblock {\em Proceedings of the Royal Society A: Mathematical, Physical \&
  Engineering Sciences}, 457:1841--1853, 2001.

\bibitem{Foer59}
H.~von Foerster.
\newblock {\em Some remarks on changing populations, In: {T}he {K}inetics of
  Cellular Proliferation}.
\newblock 382--407. Grune \& Stratton, New York, 1959.

\bibitem{Wit81}
M.~Witten.
\newblock Modeling cellular systems and aging processes. {I}. {M}athematics of
  cell system models: A review.
\newblock {\em Mechanisms of Ageing and Development}, 17:53--94, 1981.

\bibitem{ZR14}
F.~Zouyed and F.~Rebbani.
\newblock A modified quasi-boundary value method for an ultraparabolic
  ill-posed problem.
\newblock {\em Journal of Inverse and Ill-Posed Problems}, 22(4):449--466,
  2014.

\end{thebibliography}

\end{document}